\documentclass[12pt]{amsart}
\usepackage[applemac]{inputenc}
\usepackage[titletoc]{appendix}
\usepackage{bm}
\usepackage{blindtext}
\usepackage{amsmath}
\usepackage{amssymb}
\usepackage{amsfonts}
\usepackage{amsthm}
\usepackage{tikz}
\usepackage{tikz-cd}
\usepackage{caption}
\usepackage{color}
\usepackage{graphicx}
\usepackage{xcolor}
\usepackage{shadethm}
\usepackage{subfigure}

\usepackage{epigraph}

\usepackage{enumerate}
\usepackage{verbatim}
\usetikzlibrary{trees}
\usetikzlibrary{decorations.markings}
\usetikzlibrary{arrows}
\usepackage{mathrsfs}
\usepackage[margin=1.0in]{geometry}
\usepackage{xparse}

\usetikzlibrary{positioning,arrows,patterns}
\usetikzlibrary{decorations.markings}
\usetikzlibrary{calc}

\NewDocumentCommand\semiloop{O{black}mmmO{}O{above}}
{%
\draw[#1] let \p1 = ($(#3)-(#2)$) in (#3) arc (#4:({#4+180}):({0.5*veclen(\x1,\y1)})node[midway, #6] {#5};)
}

\newcommand{\Ker} {\mathrm{ker}}

\usepackage{url}
\usepackage{hyperref}
\hypersetup{colorlinks=false, linktocpage, linkbordercolor=red, citebordercolor=green, filebordercolor=red, urlbordercolor=cyan, pdfborderstyle={/S/U/W 1}}

\theoremstyle{plain}
\newtheorem{thm}{Theorem}[section]
\newtheorem{lemma}[thm]{Lemma}
\newtheorem{rmk}[thm]{Remark}

\newtheorem{prop}[thm]{Proposition}
\newtheorem{corollary}[thm]{Corollary}

\newtheorem{defn}{Definition}[section]
\newtheorem{conj}[thm]{Conjecture}

\theoremstyle{remark}

\newcommand{\R}{\mathbb{R}}

\DeclareMathOperator{\Ima}{Im}

\DeclareMathOperator{\Aut}{Aut}

\newcommand{\calH}{\mathcal{H}}

\def\gpd{\,\lower1pt\hbox{$\longrightarrow$}\hskip-.24in\raise2pt
               \hbox{$\longrightarrow$}\,}

\newcommand{\Z}{\mathbb{Z}}
\newcommand{\C}{\mathbb{C}}
\newcommand{\Cyc}{{\text{Cyc}}}

\newcommand{\supp}{\text{supp}}

\makeatletter

\pgfdeclareshape{genus}{
 \anchor{center}{\pgfpointorigin}
\backgroundpath{
    \begingroup
    \tikz@mode
    \iftikz@mode@fill

         \pgfpathmoveto{\pgfqpoint{-0.78cm}{-.17cm}}
     \pgfpathcurveto %
            {\pgfpoint{-0.35cm}{-.44cm}}
        {\pgfpoint{0.35cm}{-.44cm}}
        {\pgfpoint{.78cm}{-0.17cm}} 
     \pgfpathmoveto{\pgfqpoint{-0.78cm}{-0.17cm}}
     \pgfpathcurveto %
            {\pgfpoint{-0.25cm}{.25cm}}
        {\pgfpoint{.25cm}{.25cm}}
        {\pgfpoint{0.78cm}{-0.17cm}}
        \pgfusepath{fill}
        \fi

        \iftikz@mode@draw
     \pgfpathmoveto{\pgfqpoint{-1cm}{0cm}}
     \pgfpathcurveto %
            {\pgfpoint{-0.5cm}{-.5cm}}
        {\pgfpoint{0.5cm}{-.5cm}}
        {\pgfpoint{1cm}{0cm}}

     \pgfpathmoveto{\pgfqpoint{-0.75cm}{-0.15cm}}
     \pgfpathcurveto %
            {\pgfpoint{-0.25cm}{.25cm}}
        {\pgfpoint{.25cm}{.25cm}}
        {\pgfpoint{0.75cm}{-0.15cm}}
              \pgfusepath{stroke}
        \fi
        \endgroup
    }
    }

    \makeatother

\begin{document}

\title{Graph de Rham Cohomology and the Automorphism Group}
\author[I. Contreras]{Ivan Contreras}
\author[A. Rosevear]{Audrey Rosevear}

\email[A. Rosevear]{arosevear22@amherst.edu}

\maketitle

\begin{abstract} We introduce a graph-theoretical interpretation of an induced action of $\Aut(\Gamma)$ in the discrete de Rham cohomology of a finite graph $\Gamma$. This action produces a splitting of $\Aut(\Gamma)$ that depends on the cycles of $\Gamma$. We also prove some graph-theoretical analogues of standard results in differential geometry, in particular, a graph version of Stokes' Theorem and the Mayer-Vietoris sequence in cohomology. 
\end{abstract}


\section{Introduction}
Graphs can be interpreted as a 1-dimensional model of smooth manifolds. Functions on a graph can be modeled as elements on a finite dimensional vector space, and linear operators on graphs correspond to finite matrices. In recent years, various authors \cite{scoville,ghrist, forman,knill} have explored discrete versions of classical objects and results in differential geometry, including discrete Morse theory \cite{scoville}, discrete versions of sheaves and vector bundles \cite{ghrist}, and their corresponding topological invariants, as well as a graph theoretical version of Selberg's trace formula \cite{2007_Mnev}.

Along this direction, the first author and Xu \cite{morsecontreras} provided a graph-theoretical version of Witten's approach to Morse inequalities \cite{Witten}, using supersymmetric quantum mechanics. The key ingredient of this construction is the use of the \emph{even} and \emph{odd} versions of the graph Laplacian, and their spectral properties. 
More recently, following the formulation of quantum mechanics on graphs in \cite{2016_Mnev}, a gluing formula for the spectrum of the even and odd Laplacian have been introduced \cite{gluingcontreras}.

In this manuscript we introduce a natural action of the graph automorphism group $\Aut(\Gamma)$ in the De Rham graph cohomology for a simple graph $\Gamma$. We provide a graph-theoretical interpretation of this action (Theorem \ref{kernelinterpretation} and Theorem \ref{splittingthm}) in terms of what we call the \emph{cycle-forest decomposition}. We prove (Theorem \ref{trivialkernel}) that under certain conditions the representation of Aut$(\Gamma)$ in $H^{1}(\Gamma)$ is isomorphic to Aut$(\Gamma)$. More generally (Theorem \ref{splittingthm}) we provide a splitting theorem of the automorphism group in terms of the natural action, and in particular  we recover Theorem 3.1 in \cite{surfaceaction}. However, our results differ both in the proof method and in that our results measure the extent to which the action does act trivially on cohomology in the cases where the action is trivial.

We also provide a graph-theoretical version of the Stokes theorem, and the Mayer-Vietoris sequence of graph cohomology. We include a version of the Hodge decomposition theorem for graph discrete forms, in terms of the even and odd Laplacians.

\subsection*{Acknowledgments}
Part of this project was conducted under the Summer Undergraduate Research Fund (SURF) and the Gregory Call Fund at Amherst College. We thank Ivan Contreras for mentoring this project and Andrew Tawfeek for useful discussions.

\section*{Notation and conventions}
All graphs in this manuscript will be finite simple graphs. We denote by $C^0(\Gamma) \cong \R^{|V|}$ the vector space of formal sums of vertices of the graph $\Gamma$ with real coefficients, and likewise $C^1(\Gamma) \cong \R^{|E|}$ the vector space of formal sums of edges. Specifically, we think of edges in this space as ordered pairs of vertices $(v_i, v_j)$ with the equivalence class $(v_i, v_j) \sim -(v_j, v_i)$. We use the notation $C(\Gamma) = C^0(\Gamma) \oplus C^1(\Gamma)$. We denote by $\Omega^0(\Gamma)$, $\Omega^1(\Gamma)$, and $\Omega(\Gamma)$ the dual spaces of $C^0(\Gamma)$, $C^1(\Gamma)$, and $C(\Gamma)$, respectively. We can view $\Omega^1(\Gamma)$ as the space of antisymmetric bilinear functionals on pairs of adjacent vertices. We call elements of $\Omega^0(\Gamma)$ and $\Omega^1(\Gamma)$ \textbf{vertex forms} and \textbf{edge forms}, respectively. We define the \textbf{support} $\supp(f) \in \Omega^\bullet(\Gamma)$ of a form $f$ to be the subgraph induced by the elements on which $f$ is nonzero.

We mean by an \textbf{orientation} $\sigma \in \Omega^1(\Gamma)$ on a graph $\Gamma$ an edge form whose values on edges are $1$ or $-1$. We will think of an orientation as assigning a direction to each edge: on an edge $(v_1, v_2)$, we can think of the edge as pointing from $v_1$ to $v_2$ if $\sigma(v_1, v_2) = 1$, and pointing in the opposite direction otherwise. An orientation induces a natural basis for $\Omega^1(\Gamma)$ as follows: for each edge $e = (v_i, v_j)$ let $e(i,j) \in \Omega^1(\Gamma)$ be given by $\sigma(e)$ if $e = (v_i, v_j)$ and 0 otherwise. This forms a basis for $\Omega^1(\Gamma)$.

We define the {\textbf{coboundary operator}} $D: \Omega^0(\Gamma) \to \Omega^1(\Gamma)$ via its action on a basis edge by, if $f \in \Omega^0(\Gamma)$, $Df(v_i, v_j) = f(v_j) - f(v_i)$. This operator is linear and we have a cochain complex

\[0 \rightarrow \Omega^0(\Gamma) \xrightarrow{D} \Omega^1(\Gamma) \rightarrow 0.\]

The \textbf{cohomology groups} of $\Gamma$ are defined by $H^0(\Gamma) = \ker(D)$ and \[H^1(\Gamma) = \frac{\Omega^1(\Gamma)}{\mbox{Im}(D)} \cong \Ker(D^*),\] where $D^*$ is the adjoint of $D$ under the standard inner product. We recall the following well-known lemma:

\begin{lemma}
Let $\Gamma$ be a  graph. Then $\dim(H^0(\Gamma)) = b_0$ and $\dim(H^1(\Gamma)) = b_1$, where $b_0$, the \textbf{0th Betti number}, is the number of connected components of $\Gamma$, and $b_1$, the \textbf{1st Betti number}, is the number of independent cycles of $\Gamma$. Moreover, we can always find a basis $\alpha$ of $H^1(\Gamma)$ such that for each $f \in \alpha$, $\supp(f)$ is a cycle.
\end{lemma}

In an orientation-induced basis, the boundary operator $D^*$ is given by the incidence matrix

\[I(k,l) = \begin{cases}
-1 &\text{if $e_l$ starts at $v_k$} \\
1 &\text{if $e_l$ ends at $v_k$} \\
0 &\text{otherwise.}
\end{cases}\]

and $D$ is given by its transpose.

\section{Graph homomorphisms and pullbacks}

  We can regard graph homomorphisms as the replacement of smooth maps. In this section we will introduce graph ``pullbacks" which share many crucial properties with pullbacks in the smooth setting. Let $F: \Gamma_1 \to \Gamma_2$ be a graph homomorphism. We can identify $F$ with its ``pushforward" $F: C(\Gamma_1) \to C(\Gamma_2)$ by extending it linearly from its action on the basis $V$.

If $f \in \Omega^0(\Gamma_2)$ is a vertex form, then we define its \textit{pullback} \[F^*f: \Omega^0(\Gamma_1) \to \C\] by
\[F^*f(v) = f(Fv).
\]
For edges, we use the same definition: if $g \in \Omega^1(\Gamma_2)$ is an edge form and $(v, w) \in E_{\Gamma_1}$, we have:
\[F^*g(v, w) = g(Fv, Fw).
\]
Note while this definition is orientation independent, its representation in an edge basis induced by an orientation is dependent on the orientation.

\begin{prop}\label{pullbackprops}
Let $F: \Gamma_1 \to \Gamma_2$ and $G: \Gamma_2 \to \Gamma_3$ be graph homomorphisms.
\begin{enumerate}
    \item $F^*$ is linear.
    \item If $F: \Gamma_1 \to \Gamma_2$ and $G: \Gamma_2 \to \Gamma_3$ are graph homomorphisms, then $(G \circ F)^* = F^* \circ G^*$
    \item If $Id_A: A \to A$ is the identity map, then $(Id_A)^* = E$, where $E$ is the linear identity map.
    \item $F^*$ commutes with $D$.
\end{enumerate}
\end{prop}
\begin{proof}
The first three properties follow directly from the definitions. We prove the fourth. Let $f \in \Omega^0(\Gamma_2)$. Then
\begin{align*}F^*(Df)(v, w) &= F^*(f(w) - f(v))\\
&= f(Fw) - f(Fv) \\
\end{align*}
and
\begin{align*}
D(F^*f)(v,w) &= Df(Fv, Fw)\\
&= f(Fw) - f(Fv),
\end{align*}
as wanted.
\end{proof}

We also have the following graph-theoretical analogue of the contravariance of isomorphisms and differential forms.
\begin{lemma}
$F^*: \Omega^\bullet(\Gamma_2) \to \Omega^\bullet(\Gamma_1)$ is an isomorphism if and only if $F: \Gamma_1 \to \Gamma_2$ is a graph isomorphism.
\end{lemma}
\begin{proof}
$(\Leftarrow)$ Assume $F$ is an isomorphism. Consider $(F \circ F^{-1})^* = (Id)^*$, which implies $(F^{-1})^* \circ F^* = E$. This implies $F^*$ is injective. The reverse argument shows $(F^{-1})^*$ is injective. Since both maps are injective and their composition is the identity, they are inverses and $F^*$ is an isomorphism.

$(\Rightarrow)$ Assume $F^*$ is an isomorphism. Thus, we know the vertex and edge sets of $\Gamma_1$ and $\Gamma_2$ are equal in size. Let $v_i \in V(\Gamma_2)$. Consider the vertex $w_i = \supp(F^*(v_i)) \in V(\Gamma_1)$. By definition, we can see that $F(w_i) = v_i.$ This shows that $F$ is surjective on the vertices and thus bijective on the vertices. A similar argument shows that $F$ is bijective on the edges. Since $F$ is a bijective homomorphism, it is an isomorphism.
\end{proof}

From the lemma, we see that a graph homomorphism $F$ induces a cochain map on the cochain complexes of the $\Gamma_1$ and $\Gamma_2$. We can also see that from commutativity with $D$, $F^*$ descends into a well-defined linear map on cohomology (the proof is exactly analogous to, for example, Proposition 11.1 in \cite{Lee}), which we also denote $F^*$.

\begin{lemma}
If $F^*: \Omega^\bullet(\Gamma_2) \to \Omega^\bullet(\Gamma_1)$ is an isomorphism, then the cohomology map $F^*: H^\bullet(\Gamma_2) \to H^\bullet(\Gamma_1)$ is an isomorphism.
\end{lemma}
\begin{proof}
Since $F^*: \Omega^\bullet(\Gamma_2) \to \Omega^\bullet(\Gamma_1)$ is an isomorphism, $F$ is an isomorphism, so $H^\bullet(\Gamma_1) \cong H^\bullet(\Gamma_2)$. We need only show that the kernel is trivial. Let $F^*[f] = [0]$, which is equivalent to $F^*f = D\alpha$ for some $\alpha$. Since $F^*: \Omega^\bullet(\Gamma_2) \to \Omega^\bullet(\Gamma_1)$ is an isomorphism, $\alpha = F^*\beta$ for some $\beta$. So $F^*f = DF^*\beta = F^*D\beta$ so $f = D\beta$ and $[f] = [0]$.
\end{proof}

From the preceding discussion, we also have a proof of the following known result:

\begin{corollary}
The map $\Gamma \mapsto H^\bullet(\Gamma)$ is a contravariant functor from the category $\text{Gr}$ of graphs and graph homomorphisms to the category $\text{Vect}$ of vector spaces and linear maps.
\end{corollary}

\section{Main Results: Induced action of Aut$(\Gamma)$}
We now turn our attention to a finite simple graph $\Gamma$ and its automorphism group $\Aut(\Gamma)$. We identify each element of $\Aut(\Gamma)$ with a permutation matrix on $C^0(\Gamma)$. By the preceding lemmas, this group of permutation matrices descends into two well-defined groups of linear operators on $H^0(\Gamma)$ and $H^1(\Gamma)$ via their pullbacks, respectively. We will denote these groups $\calH^0(\Gamma)$ and $\calH^1(\Gamma)$, respectively.

The groups each come with a group homomorphism, denoted $\phi^0_\Gamma$ and $\phi^1_\Gamma$, respectively. We will suspend the sub- and superscripts when there is no ambiguity as to which induced group we are talking about. Both homomorphisms are given by \[\phi(g) = (g^*)^{-1},\] where $g \in \Aut(\Gamma)$  (the inversion makes $\phi$ a into homomorphism). By construction, $\phi$ is surjective, so we can focus on calculating its kernel to compute $\calH^0(\Gamma)$ and $\calH^1(\Gamma)$.

\subsection{Graph-theoretical interpretation of $\calH^0$}

The following proposition describes $\calH^0(\Gamma)$ in terms of the connected components of $\Gamma$.

\begin{prop}\label{calH0interpretation}
Let $g \in \Aut(\Gamma)$. Then $g \in \ker(\phi)$  if and only if, for all $v \in V$, $v$ and $g(v)$ lie in the same connected component of $\Gamma$, i.e. $g$ preserves connected components.
\end{prop}
\begin{proof}
($\Leftarrow$) Let $g$ preserve connected components of $\Gamma$. Let $\omega \in H^0(\Gamma)$. Then it is constant on each connected component of $\Gamma$, so $g^*\omega(v) = \omega(v)$ for all $v \in V$.\\
($\Rightarrow$) Let $g \in \ker(\phi)$. If $\Gamma$ has connected components $\Gamma_1, ... \Gamma_n$, then we can construct a basis $\{\omega_i\}_{1 \leq i \leq n}$ for $H^0(\Gamma)$ given by 
\[\omega_i(v) = \begin{cases}
1 & v \in \Gamma_i\\
0 & \text{otherwise}\\
\end{cases}.\]
Then $g \in \ker(\phi)$ means $g^*\omega_i(v) = \omega_i(g(v)) = \omega_i(v)$ for all $i$ and $v$. This implies $v$ and $g(v)$ are on the same connected component.
\end{proof}

\begin{corollary}
The graph $\Gamma$ has at least two isomorphic connected components if and only if $\calH^0(\Gamma)$ is nontrivial.
\end{corollary}

Though a fairly simple interpretation, this corollary will help us to compute $\calH^1(\Gamma)$ when a certain subgraph, $F(\Gamma)$, defined in the next section, satisfies triviality of $\calH^0(F(\Gamma))$.

\begin{figure}[h]
    \centering
    \includegraphics[scale = .45]{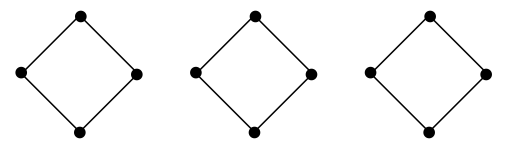}
    \caption{A graph $\Gamma$ with $\Aut(\Gamma) \cong D_8 \wr S_3$ and $\calH^0(\Gamma) \cong S_3$.}
    \label{H0ex}
\end{figure}

\subsection{Cycle-Forest Decompositions}

Before moving on to $\calH^1(\Gamma)$, we introduce a decomposition of graphs which will provide a valuable tool both for proving and interpreting our results on $\calH^1(\Gamma)$.

We construct the decomposition for graphs with $b_1 \geq 1$ as follows. Let $F_1$ be the set of all the leaves of $\Gamma$. Recursively define $F_j$ as the set of all leaves of $\Gamma - \bigcup^{j-1}_{i = 1}F_i$. Eventually, we must have some $k$ such that $F_j$ is empty for all $j > k$. Then $\Gamma - \bigcup^{k}_{i = 1}F_i$ will have minimum valency either 0 or at least 2. Let $F(\Gamma)$ be the subgraph induced by $\bigcup^{k}_{i = 1}F_i$, including all edges incident to $F_k$, and let \[\Cyc(\Gamma) = \Gamma -  \bigcup^{k}_{i = 1}F_i.\] Clearly $F(\Gamma) \cup \Cyc(\Gamma) = \Gamma$. We call this decomposition the \textit{cycle-forest decomposition} of $\Gamma$. We call $F(\Gamma)$ the \textit{contractible forest} of $\Gamma$, and we call $\Cyc(\Gamma)$ the \textit{cycle retract} of $\Gamma$. Note that either the contractible forest or the cycle retract may be empty. We call $\Cyc(\Gamma) \cap F(\Gamma)$ the \textit{intersection vertices}. If $\Gamma$ is a tree or a forest, we take the convention that $F(\Gamma) = \Gamma$ and $\Cyc(\Gamma) = \emptyset$.

\begin{prop}
Let $\Gamma$ be a finite, simple graph.
\begin{enumerate}
    \item If $F(\Gamma)$ is nonempty, then it is a forest.
    \item If $\Cyc(\Gamma)$ is nonempty, then it is a retract of $\Gamma$.
\end{enumerate}
\end{prop}
\begin{proof}
The first part follows directly from construction. We prove the second part.\\
Clearly there exists a homomorphism $\Cyc(\Gamma) \to \Gamma$, the inclusion homomorphism. We construct a homomorphism $\Gamma \to \Cyc(\Gamma)$. Let $h$ be a map $\Gamma \to \Cyc(\Gamma)$ defined as follows. If $v \in \Cyc(\Gamma)$, let $h(v) = v$. Let $A$ be a connected component of $F(\Gamma)$. Since $A$ is a tree, it is bipartite. Moreover, it connects to a unique vertex $v \in \Cyc(\Gamma)$. Let $(v, w)$ be an edge connected to $v$ in $\Cyc(\Gamma)$. Since $A$ is bipartite, the we may extend $h$ to a homomorphism $A \to (v,w)$ that fixes $v$. Doing this for all connected components of $F(\Gamma)$, $h$ becomes a homomorphism $\Gamma \to \Cyc(\Gamma)$. Clearly the composition of $h$ and the inclusion homomorphism is the identity, so $\Cyc(\Gamma)$ is a retract of $\Gamma$.
\end{proof}

We can think of this as a contraction on $\Gamma$ akin to a homotopy retraction on manifolds. This particular decomposition is in a way the minimal such retraction, as the following makes precise:

\begin{prop}\label{Cychomotopy}
$\Cyc(\Gamma)$ is the unique smallest subgraph of $\Gamma$ such that $H^\bullet(\Cyc(\Gamma)) = H^\bullet(\Gamma)$.
\end{prop}
\begin{proof}
By construction, $\Cyc(\Gamma)$ and $\Gamma$ have the same number of connected components. Moreover, since $F(\Gamma)$ is a forest, $\Cyc(\Gamma)$ and $\Gamma$ have the same number of cycles. Therefore $\Cyc(\Gamma)$ and $\Gamma$ are cohomology equivalent.\\
Let $\alpha = \{c_1, ..., c_n\}$ be the supports of a basis for $H^1(\Cyc(\Gamma))$. Consider removing an edge $e$ from $\Cyc(\Gamma)$. If $e \in \bigcup_{c_i \in \alpha}c_i$, then removing it would decrease $H^1(\Cyc(\Gamma))$. If $e \notin \bigcup_{c_i \in \alpha}c_i$, then any path starting at $e$ and going in either direction must eventually intersect a cycle, else $e$ would be in the contractible forest. Since $e$ is not part of a basis cycle, it must lie in a connecting bridge between cycles, and thus removing it would increase $H^0(\Cyc(\Gamma))$. Thus $\Cyc(\Gamma)$ cannot be any smaller without changing the cohomology.\\
To show $\Cyc(\Gamma)$ is the unique smallest subgraph, consider building the smallest subgraph which is cohomology equivalent to $\Gamma$. It must contain all of $\alpha$. Moreover, if $\bigcup_{c_i \in \alpha}c_i$ is not connected, any two connected components can by hypothesis be connected by at most one connecting bridge. Adding any such connecting bridges to $\bigcup_{c_i \in \alpha}c_i$, we recover $\Cyc(\Gamma)$.
\end{proof}

\begin{figure}
    \centering
    \includegraphics[scale = .45]{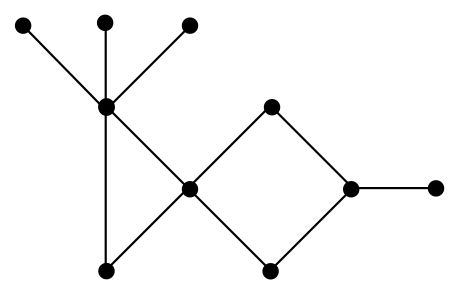}\linebreak
    \includegraphics[scale = .45]{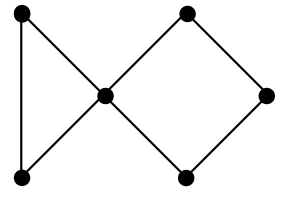} \qquad \qquad
    \includegraphics[scale = .45]{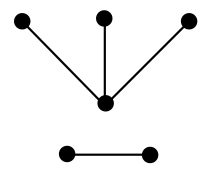}
    \caption{Above: a graph $\Gamma$. Left: $\Cyc(\Gamma)$. Right: $F(\Gamma)$. We also have $\calH^0(\Gamma) \cong \{e\}$ and $\calH^1(\Gamma) \cong \Z/2\Z$ by Theorem \ref{splittingthm}.}
    \label{cycleforestex}
\end{figure}

This decomposition is also compatible with the automorphism group.

\begin{prop}
Let $g \in \Aut(\Gamma)$. Then $g$ restricts to an automorphism on both $F(\Gamma)$ and $\Cyc(\Gamma)$.
\end{prop}
\begin{proof}
Since automorphisms preserve valency, g must map elements of $F_1$ from the construction of $F(\Gamma)$ to other elements of $F_1$. Inductively, if $g$ restricts to an automorphism on $F_j$, it must also restrict to an automorphism on $F_{j+1}$. Thus $g$ restricts to an automorphism on $\bigcup^{k}_{i = 1}F_i = F(\Gamma)$. Since $g$ is bijective, it must also restrict to an automorphism on $\Cyc(\Gamma)$.
\end{proof}

\begin{corollary}
Let $g \in \Aut(\Gamma)$. Then $g$ permutes the intersection vertices.
\end{corollary}

\subsection{Graph-theoretical interpretation of $\calH^1$}
A graph-theoretical interpretation of $\calH^1(\Gamma)$ is somewhat more involved than the $\calH^0(\Gamma)$ case. We prove a series of results culminating in the Automorphism Splitting Theorem for graphs satisfying triviality of $\calH^0(F(\Gamma))$. We begin with a simple technical proposition, which follows easily from the definitions:

\begin{prop}
If $g \in \Aut(\Gamma)$, $[f] \in H^1(\Gamma)$, then $g^*[f] = [f]$ if and only if $\langle g^*f - f, w \rangle = 0$ for all $w \in \Ima(D) = \ker(I)^\perp$.
\end{prop}

Our first step towards a characterization of $\calH^1(\Gamma)$ begins with the following useful lemma, characterizing the kernel of the surjective homomorphism $\phi^1$. From this point we omit the superscript. This lemma is the basis of all of our major results interpreting $\calH^1(\Gamma)$.

\begin{lemma}\label{fundamentallemma}
Let $\Gamma$ be a graph, and let $g \in \Aut(\Gamma)$. Then $g \in \ker(\phi)$ if and only if $g$ restricts to a rotation on each primitive cycle of $\Gamma$.
\end{lemma}
\begin{proof}
($\Leftarrow$) Assume $g$ restricts to a rotation on each primitive cycle. Fix some basis $\{f_i\}$ for $H^1(\Gamma)$, considered as a subspace of $\Omega^1(\Gamma)$, such that $\supp(f_i)$ is a cycle for each $f_i$. Consider one particular basis element $f$ with support having length $k$. We can fix an edge basis such that $f = \sum_{i = 1}^{k-1} (v^i, v^{i+1})$. Then for any edge $(v_j, v_{j+1}) \in \supp(f)$, we can see $f(v_j, v_{j+1}) = 1$. Since $g$ is a rotation, we have for some fixed integer $n$, $g^*f(v_j, v_{j+1}) = f(g(v_j), g(v_{j+1}) = f(v_{(j + n) \mod{k}}, v_{(j + 1 + n) \mod{k}}) = 1$. Moreover, if $(w_j, w_l) \notin \supp(f)$, then $f(w_j, w_l) = 0 = g^*f(w_j, w_l)$. Thus $g^*f = f$. Since $f$ was an arbitrary basis element, $g^*f_j = f_j$ for all $j$, and $g^*$ fixes $H^1(\Gamma)$, so $g^*$ is the identity on $H^1(\Gamma)$ and $g \in \ker(\phi)$.\\
($\Rightarrow$) Assume $g \in \ker(\phi)$, and let $\{f_i\}$ again be a basis for $H^1(\Gamma)$ such that $\supp(f_i)$ is a cycle for each $f_i$. Then for all $\{f_i\}$, $g^*[f_i] = [f_i]$. Fix such an $f_i$. This means that $\langle g^*f_i - f_i, w \rangle = 0$ for all $w \in \ker(I)$. This means $\langle g^*f_i - f_i, f_i \rangle = \langle g^*f_i, f_i \rangle - \langle f_i, f_i \rangle = 0$ or $\langle g^*f_i, f_i \rangle = n$, where $\supp(f_i) \cong C_n$. But since $\supp(g^*f_i) \cong C_n$, and both $f_i$ and $g^*f_i$ only take values 1, -1, or 0, the only way for this to happen is to have $g^*f_i = f_i$, which means that $g$ must be a rotation on $\{\supp(f_i)\}$.
\end{proof}

This comes with an useful corollary:

\begin{corollary}
Let $\Gamma$ be a graph, and let $g \in \ker(\phi_\Gamma)$. Let $g'$ be the automorphism on $\Cyc(\Gamma)$ given by restricting $g$ to $\Cyc(\Gamma)$. Then $g' \in \ker(\phi_{\Cyc(\Gamma)})$.
\end{corollary}
\begin{proof}
By the lemma, $g$ restricts to a rotation on all cycles of $\Gamma$. Thus $g'$ restricts to a rotation on all cycles of $\Cyc(\Gamma)$, and thus $g' \in \ker(\phi_{\Cyc(\Gamma)})$.
\end{proof}

From the lemma we can easily compute $\calH^1(C_n)$:

\begin{corollary}
$\calH^1(C_n) \cong \Z/2\Z.$
\end{corollary}

The case $b_1 = 1$ turns out to be quite different from the cases $b_1 \geq 2$. In the first case, there are only two options for $\calH^1(\Gamma)$. Since $\calH^1(\Gamma)$ is a finite group of linear operators, all of its elements have determinant 1 or -1. Since $b_1 = 1$, these linear operators are 1-dimensional, so there at most 2 elements in the group. Thus $\calH^1(\Gamma)$ is either trivial or cyclic of order 2.

There are many cases in which $\calH^1(\Gamma) \cong \Z/2\Z$ when $b_1 = 1$. We list three. The main property that determines this is when the structure of $F(\Gamma)$ and the intersection vertices allows for an orientation-reversing automorphism on the sole cycle of $\Gamma$.

\begin{prop}\label{b1=1}
Let $b_1 = 1$, and let $\Gamma$ be a graph. Then $\calH^1(\Gamma) \cong \Z/2\Z$ if one of the following is true:
\begin{enumerate}
    \item $F(\Gamma)$ is empty.
    \item $F(\Gamma)$ is connected.
    \item $F(\Gamma)$ has precisely 2 connected components, $\Cyc(\Gamma) \cong C_{2n}$ for some $n$, and the intersection vertices are a distance of precisely $n$ apart.
\end{enumerate}
\end{prop}
\begin{proof}
If $F(\Gamma)$ is empty, $\Gamma \cong C_n$ for some $n$ and the results follows from the previous corollary. If $F(\Gamma)$ is connected, then there is precisely 1 intersection vertex. The automorphism group then splits into a product of automorphisms which fix the cycle and thus are in $\ker(\phi)$, and a single orientation-reversing automorphism $g$ which fixes $F(\Gamma)$. Since it is orientation reversing, $g^* = -1$, so $\Ima(\phi) = \{1, -1\} \cong \Z/2\Z$. The third case is similar, since the spacing of the intersection vertices allows for the orientation reversing automorphism on the cycle.
\end{proof}

We now move on to the case $b_1 \geq 2$. We prove our first main theorem characterizing $\calH^1(\Gamma)$ for a large class of graphs, namely the ones for which $\Cyc(\Gamma) = \Gamma$. This result is very similar to Theorem 3.1 in \cite{surfaceaction}, but with a different, combinatorial proof.

\begin{thm}\label{trivialkernel}
Let $\Gamma$ be a connected graph with minimum valency $\geq 2$ and $b_1 \geq 2$. Then $\calH^1(\Gamma) \cong \Aut(\Gamma)$, i.e. every automorphism has a nontrivial action on cohomology.
\end{thm}
\begin{proof}
We will prove the theorem by induction on $b_1$.

{\bf{Base Case.}}\\
Let $\Gamma$ be a graph satisfying the above conditions with $b_1 = 2$. Let $g \in \ker(\phi)$. Let $C_1$ and $C_2$ be the two primitive cycles, considered as unordered sets; by the previous lemma, $g$ restricts to a rotation on each. We have two cases:

{\emph{Case 1: $C_1$ and $C_2$ overlap on at least one edge.}}

Call the overlapping edge set $E' =\{e_1, ... e_n\}$. In this case, since $g \in \ker(\phi)$, $g(C_1) = C_1$, and $g_(C_2) = C_2$, which means $g(E') = E'$. Moreover, $E'$ must be connected as a subgraph since $b_1 = 2$. Since $g$ is orientation preserving, it therefore must fix $E'$. Since $g$ is a rotation on both $C_1$ and $C_2$, since it fixes at least one edge in each it must fix both $C_1$ and $C_2$. Since the minimum degree of the graph is greater than or equal to 2, $C_1 \cup C_2 = \Gamma$, so $g$ fixes $\Gamma$ and thus is the identity automorphism.

{\emph{Case 2: $C_1$ and $C_2$ do not overlap.}}

In this case, $C_1$ and $C_2$ must be connected by precisely one bridge, or overlap at a single vertex, since more than one bridge or vertex would introduce more cycles. Let $v_1$ and $v_2$ be the vertices of $C_1$ and $C_2$, respectively, that are either the overlapping vertices or connected by this bridge. Then $g$ must fix both of these vertices. Since it is a rotation on $C_1$ and $C_2$, a fixed point implies again that it fixes both cycles. Therefore it also fixes the bridge (if it exists) and therefore is the identity automorphism.

{\bf{Inductive Step.}}\\
Assume the theorem is true in the case $b_1(\Gamma) = n$. Let $\Gamma$ have the conditions listed above with $b_1(\Gamma) = n + 1$. Let $e$ be an edge of $\Gamma$ that lies in some primitive cycle $C$. Then $b_1(\Gamma - \{e\}) = n$. However, $\Gamma - \{e\}$ may not have minimum valency 2. We now use the inductive process from the definition of the cycle-forest decomposition, taken on $\Gamma - \{e\}$. Let $F_1$ be the set of leaves of $\Gamma - \{e\}$, and let $F_j = \Gamma - \bigcup^{j-1}_{i = 1}F_i$ as previous, with the process ending on some $F_k$. Then $\Cyc(\Gamma - \{e\})$ will satisfy the inductive hypothesis. We now inductively add back the removed vertices, starting with $F_k$. We have two cases.

{\emph{Case 1: None of the vertices in $|F_j| = 1$ for all $j$}}.

In this case we may work backwards as follows. Since $g$ fixes $\Gamma - \{e\} - \bigcup^{k}_{i = 1}F_i$, it must fix $\Gamma - \{e\} - \bigcup^{k-1}_{i = 1}F_i$, since we are only readding a single edge and vertex. Likewise for any $j$, $g$ fixes $\Gamma - \{e\} - \bigcup^{j-1}_{i = 1}F_i$, and thus $g$ fixes $\Gamma - \{e\}$, so $g$ fixes $\Gamma$ and is thus the identity automorphism.

{\emph{Case 2: There exists some $j$ such that $|F_i| = 1$ for $i > j$, and $|F_i| = 2$ for $i \leq j$.}}

The fact that this is the only other case follows from the fact that the minimum valency of $\Gamma$ is 2. In this case, the same argument as the previous case implies that $g$ fixes $\Gamma - \{e\} - \bigcup^{j+1}_{i = 1}F_i$. Moreover, we have that $\{e\} \cup (\bigcup_{i = j}^k F_i)$ (taken as the subgraph induced by the vertices) is a primitive cycle, since it is a cycle which does not intersect any other cycles. Since $g \in \ker(\phi)$, $g$ is a rotation on this cycle. But $g$ fixes at least one vertex in this cycle: the vertex $F_{j+1}$. Since $g$ is a rotation which fixes a vertex, it must fix the entire cycle, and thus $g$ fixes $\Gamma$ and is the identity automorphism.
\end{proof}

Though this result tells us that our entire construction is trivial for this large class of graphs, we can still use it to extract information in cases when $\calH^1(\Gamma)$ is not isomorphic to $\Aut(\Gamma)$.

\begin{corollary}
Let $\Gamma$ have $b_1 \geq 2$, and let $g \in \ker(\phi)$. Then $g$ fixes $\Cyc(\Gamma)$.
\end{corollary}
\begin{proof}
Let $g \in \ker(\phi_\Gamma)$. Let $g' \in \ker(\phi_{\Cyc(\Gamma)})$ be given by restricting $g$ to $\Cyc(\Gamma)$. By the previous theorem, $g'$ fixes $\Cyc(\Gamma)$, and thus $g$ fixes $\Cyc(\Gamma)$.
\end{proof}

We can harness this to get a purely graph-theoretic characterization of $\ker(\phi)$.

\begin{thm}\label{kernelinterpretation}
Let $S \leq \Aut(F(\Gamma))$ be the subgroup of $\Aut(F(\Gamma))$ which fixes the intersection vertices $F(\Gamma) \cap \Cyc(\Gamma)$. Then if $b_1 \geq 2$ and $\Gamma$ is connected, $\ker(\phi) \cong S$.
\end{thm}
\begin{proof}
Let $s \in S$, and let $\eta_s$ be the automorphism on $\Gamma$ obtained by fixing $\Cyc(\Gamma)$ and acting by $s$ on $F(\Gamma)$. Let
\[\eta: S \to \ker(\phi)\]
be given by $\eta(s) = \eta_s$. Since the image of $\eta$ fixes the cycles of $\Gamma$, it indeed lies in $\ker(\phi)$. It is clear that $\eta$ is a homomorphism, since its action on the product of two automorphism will simply be the product of those two automorphisms both extended to the identity on $\Cyc(\Gamma)$. Let $s \in \ker(\eta)$. Then $\eta(s)$ is the identity and fixes all of $\Gamma$. It therefore fixes $F(\Gamma)$, so $s$ is the identity automorphism on $F(\Gamma)$ and $\eta$ is injective. Let $g \in \ker(\phi)$. By the previous corollary, it fixes $\Cyc(\Gamma)$ and therefore restricts to an automorphism $a$ on $F(\Gamma)$ which fixes the intersection vertices. Then $\eta(a) = g$ and $\eta$ is surjective, and thus an isomorphism.
\end{proof}

The two previous theorems provide us with our first intuitive characterization of $\calH^1(\Gamma)$, which is that it is measuring  ``how much" of $\Aut(\Gamma)$ lives in $\Cyc(\Gamma)$, while $\ker(\phi)$ is measuring how much of $\Aut(\Gamma)$ lives in $F(\Gamma)$. In the case that $\calH^0(F(\Gamma))$ is trivial, we can make this more precise with our following main result.

\begin{thm}\label{splittingthm}{(Automorphism Splitting Theorem)}
If $\Gamma$ is connected with $b_1 \geq 2$, and $\calH^0(F(\Gamma))$ is trivial, then 
\[\Aut(\Gamma) \cong \ker(\phi^1_\Gamma) \times \calH^1(\Gamma),\]
where each element of $\ker(\phi)$ corresponds to an automorphism that fixes $\Cyc(\Gamma)$, and each element of $\calH^1(\Gamma)$ corresponds to an automorphism that fixes $F(\Gamma).$
\end{thm}

\begin{proof}
Since $\calH^0(F(\Gamma))$ is trivial, all automorphisms of $\Gamma$ preserve connected components of $F(\Gamma)$, and so all elements of $\Aut(\Gamma)$ fix the intersection vertices of $\Gamma$. Thus every element $g$ of $\Aut(\Gamma)$ can be written uniquely as the product $g = xy$, where $x \in \ker(\phi)$ fixes $\Cyc(\Gamma)$ and $y$ fixes $F(\Gamma)$. Let $Y$ be the subgroup of $\Aut(\Gamma)$ consisting of all automorphisms which fix $F(\Gamma)$. Then any element in $\ker(\phi)$ commutes with any element of $Y$. From this is follows that $\Aut(\Gamma) \cong \ker(\phi) \times Y$. 
We now claim $\calH^1(\Gamma) \cong Y$. For ease of notation, we denote $\ker(\phi) = K$. Note $\calH^1(\Gamma) \cong \Aut(\Gamma)/K$. Let $g = xy \in \Aut(\Gamma)$, $xy$ as above. Then the cosets $Kxy$ and $Ky$ are equal. Thus we have a well defined map 
\[\theta: \Aut(\Gamma)/\ker(\phi) \to Y\] 
given by $\theta(Kxy) = y$. This map is homomorphism by a similar computation as for $\eta$. It is injective because, if $\theta(Kxy) = e$, then $y = e$, so $Kxy = Kx = Ke$ is the identity. We can also see that for any $y \in Y$, $\theta(Ky) = y$, so $\theta$ is surjective and an isomorphism.
\end{proof}

\begin{rmk} 
Theorem \ref{splittingthm} does not hold without the assumption that $\Gamma$ is connected and $\calH^0(F(\Gamma))$ is trivial, since in that case the automorphism group has more complicated structure. Notably, this means the intersection vertices are not fixed which prevents from making the factorization $\Aut(\Gamma) \cong \ker(\phi) \times Y$.
\end{rmk}

\section{Graph De Rham calculus}
In this section we introduce the theory of integration over graphs. The graph integration theory closely parallels the theory of integration over manifolds and the de Rham derivative. We define two integrals: one for functions defined on the vertices of a graph (0-forms) and one for functions defined on the edges (1-forms). We then prove a discrete Stokes' Theorem.\\

These definitions are orientation dependent, unlike the other results in this paper.  We therefore define integration with respect to an orientation. In the case of integration over a path, the Stokes' Theorem reduces to the well-known graph Stokes' Theorem. Thus our results generalize known results at the cost of more serious dependence on orientation.

We begin with some preliminary definitions. We define the \textit{net degree} $n_\sigma^\Gamma(v)$ of a vertex $v$ in a graph $\Gamma$ with respect to an orientation $\sigma$ as

\[n_\Gamma^\sigma(v) = \sum_{w \in \Gamma,\, w \sim v} \sigma(w, v)\]

This definition is well defined on subgraphs of $\Gamma$ containing $v$. If a subgraph $G$ of $\Gamma$ does not contain $v$, we define $n_G^\sigma(v)$ to be 0. 

We define the intersection of two subgraphs to be the set of all vertices and edges in both subgraphs (as opposed to the induced subgraph of the common vertices). This intersection will still be a graph, as it will still contain all of the vertices of its edges. We define the union similarly. Note that if $G_1$ and $G_2$ are edge-disjoint subgraphs of $\Gamma$, then \[n_{G_1}^\sigma(v) + n_{G_2}^\sigma(v) = n_{G_1\cup G_2}^\sigma(v).\] If $G_1$ and $G_2$ are not edge-disjoint, then \[n_{G_1}^\sigma(v) + n_{G_2}^\sigma(v) = n_{G_1\cup G_2}^\sigma(v) + n_{G_1 \cap G_2}^\sigma(v).\]

We denote by $-\Gamma$ the graph obtained by multiplying the orientation form by -1, i.e. switching the orientation of every edge on $\Gamma$.

\begin{defn}
Let $\Gamma$ be a graph with orientation $\sigma$. Let $f \in \Omega^0(\Gamma)$ be a vertex form on $\Gamma.$ The {\bf{vertex integral}} of $f$ over $\Gamma$ is defined to be
\[ \int^+_{\Gamma}f = \sum_{v_i \in \Gamma}f(v_i)n_\Gamma^\sigma(v_i).\]
\end{defn}

\begin{defn}
Let $\Gamma$ be a graph. Let $f \in \Omega^1(\Gamma)$ be a discrete function on the edges of $\Gamma.$ The {\bf{edge integral}} of $f$ over $\Gamma$ is defined to be 
\[ \int^-_{\Gamma}f = \sum_{(v_i, w_i) \in \Gamma}f(v_i, w_i)\sigma(v_i, w_i).\]
\end{defn}

\begin{rmk}
When it is not specified that an integral is a vertex or an edge integral, then the result holds for both types.
\end{rmk}

We define integrals over subgraphs of $\Gamma$ by simply replacing the sum with a sum over the elements of the subgraph, and changing the net degree to be with respect to the subgraph.

\begin{lemma}\label{integralprops} (Properties of the Graph Integral) Let $f$ and $g$ be discrete functions on $\Gamma$, and let $c \in \R$. Let $G_1$ and $G_2$ be subgraphs of $\Gamma$. Then the following are true:
\begin{enumerate}
    \item (Linearity) $\int_{G_1} cf + g = c\int_{G_1} f + \int_{G_1} g.$
    \item $-\int^+_{G_1} f = \int^+_{-G_1} f$.
    \item \[\int_{G_1}f + \int_{G_2}f = \int_{G_1 \cup G_2}f + \int_{G_1 \cap G_2}f,\]
    where the union and intersection are taken as the union and intersection of vertices and edges.
\end{enumerate}
\end{lemma}
\begin{proof}
Part (1) follows directly from the definitions. Part (2) follows from the fact that $n^{-\sigma}_G(v) = -n^\sigma_G(v)$.

For part (3), we prove only the vertex integral case; the edge integral case is similar, but simpler. If $G_1 \cap G_2 = \emptyset$, then the results follows directly. Let $G_3 = G_1 \cap G_2$, and assume $G_3$ is nonempty. Denoting by $\Delta$ the symmetric difference, we have
\begin{align*}
\int^+_{G_1}f + \int^+_{G_2}f &= \sum_{v_i \in G_1}f(v_i)n_{G_1}^\sigma(v_i) + \sum_{v_i \in G_2}f(v_i)n_{G_2}^\sigma(v_i)\\
&= \sum_{v_i \in G_1 \Delta G_2}f(v_i)n_{G_1 \cup G_2}^\sigma(v_i) + \sum_{v_i \in G_3}\big[f(v_i)n_{G_1}^\sigma(v_i)+f(v_i)n_{G_2}^\sigma(v_i)\big]\\
&= \sum_{v_i \in G_1 \Delta G_2}f(v_i)n_{G_1 \cup G_2}^\sigma(v_i) + \sum_{v_i \in G_3}\big[f(v_i)n_{G_1 \cup G_2}^\sigma(v_i)+f(v_i)n_{G_3}^\sigma(v_i)\big]\\
&= \sum_{v_i \in G_1 \cup G_2}f(v_i)n_{G_1 \cup G_2}^\sigma(v_i) + \sum_{v_i \in G_3}f(v_i)n_{G_3}^\sigma(v_i)\\
&= \int^+_{G_1}f + \int^+_{G_2}f\\
&= \int^+_{G_1 \cup G_2}f + \int^+_{G_3}f,
\end{align*}
as desired.
\end{proof}

\begin{thm}\label{stokesthm}
(Stokes' Theorem for Graphs) Let $f$ be a function defined on the vertices of $\Gamma$. Let $D$ be the incidence matrix of $\Gamma$. Then \[\int^+_\Gamma f = \int^-_\Gamma D f. \]
\end{thm}


\begin{proof}
Without loss of generality, we may consider only a connected graph with at least one edge. We may decompose G into an edge-disjoint union of subgraphs $G = \bigcup_{i = 1}^N G_i$ with each $G_i \cong P_2$. For each $G_i$, denote its vertices $v_1$ and $v_2$. We have
\begin{align*}
    \int^+_{G_i} f &= f(v_1)\sigma(v_1, v_2) + f(v_2)\sigma(v_1, v_2)\\
    &= \sigma(v_2, v_1)(f(v_1) - f(v_2))\\
    &= \sigma(v_2, v_1)Df(v_2, v_1)\\
    &= \int^-_{G_i} Df.
\end{align*}
Since $\{G_i\}$ is an edge-disjoint partition of $G$, we have by Lemma \ref{integralprops}
\[\int^+_G f = \int^+_{\bigcup_{i = 1}^N G_i} f = \sum_{i = 1}^N \int^+_{G_i} f = \sum_{i = 1}^N\int^-_{G_i} D f = \int^-_G D f,\]
as wanted.
\end{proof}

\begin{figure}
    \centering
    \includegraphics[scale = .45]{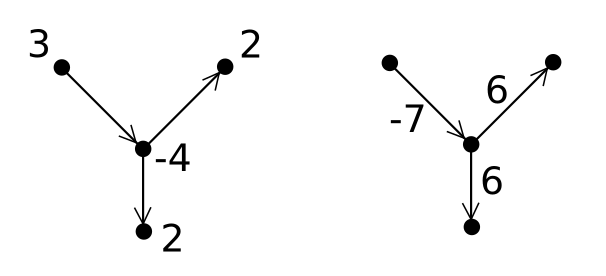}
    \caption{A vertex form $f$ and its derivative $Df$. The integral of both is 5 with respect to the orientation shown.}
    \label{stokesex}
\end{figure} 

This discussion also leads to an analogue of a classic result, based a simple linear algebra result. Recall that the {\textit{even graph Laplacian}}, frequently referred to as simply the graph Laplacian, is given by \[\Delta^+ = D^*D: \Omega^0(\Gamma) \to \Omega^0(\Gamma),\] and the {\textit{odd graph Laplacian}} is given by \[\Delta^- = DD^*: \Omega^1(\Gamma) \to \Omega^1(\Gamma).\] We have by basic linear algebra that \[\ker(D) = \ker(D^*D) = \ker(\Delta^+),\]and also \[\ker(D^*) = \ker(DD^*) = \ker(\Delta^-).\] The decomposition of $\Omega(\Gamma)$ into the kernels and images of these operators yields the following analogue of the Hodge Decomposition:

\begin{thm}\label{hodgedecomp} (Hodge Decomposition for Graph Forms)
Let $\Gamma$ be an oriented graph. Then \[\Omega^0(\Gamma) = \ker(\Delta^+) \oplus \Ima(I)\] and \[\Omega^1(\Gamma) = \ker(\Delta^-) \oplus \Ima(D)\].
\end{thm}

\section{Graph Mayer-Vietoris sequences}
We now use graph pullbacks to develop a Mayer-Vietoris Sequence for graphs. Due to the similarity of the properties of the graph pullback and the smooth pullback, this construction follows almost exactly like the smooth case, see \cite{Lee} or \cite{bottandtu} Let $\Gamma$ be a graph, and $A$ and $B$ subgraphs of $\Gamma$ such that $A \cup B = \Gamma$. We have the sequence

\begin{center}\begin{tikzcd}
  A \cap B \arrow[r, shift left, "{i}"]
        \arrow[r, shift right, "{j}"']
    & A \sqcup B 
        \arrow[r, "{(k, l)}"]
& \Gamma
\end{tikzcd}\end{center}

where $k, l, i,$ and $j$ are inclusion homomorphisms and $A \sqcup B$ is the disjoint union. Taking pullbacks, we can construct the sequence

\begin{align}\label{MVOmega}
0 \rightarrow \Omega^p(\Gamma) \xrightarrow{k^* \oplus l^*} \Omega^p(A) \oplus \Omega^p(B) \xrightarrow{i^* - j^*} \Omega^p(A \cap B) \rightarrow 0
\end{align}
\\
Where we define $(i^* - j^*)(f, g) = i^*f - j^*g$.

\begin{thm}\label{mayervietoris}(Mayer-Vietoris Sequence for Graphs)
Let $\Gamma$ be a directed graph. Let $A$ and $B$ be two subgraphs of $\Gamma$ whose union is $\Gamma$. Then there exists an operator $\delta$ such that the following sequence is exact:
\begin{align*}
0 \rightarrow H^0(\Gamma) \xrightarrow{k^* \oplus l^*} H^0(A) \oplus H^0(B) & \xrightarrow{i^{*}-j^{*}} H^0(A \cap B) \\
 \xrightarrow{\delta} H^1(\Gamma) \xrightarrow{k^{*} \oplus l^{*}} H^1(A) \oplus H^1(B) & \xrightarrow{i^{*}-j^{*}} H^p(A \cap B) \rightarrow 0
\end{align*}
where $i, j, k$ and $l$ are inclusion maps as in (\ref{MVOmega}).
\end{thm}
\begin{proof}
As in the smooth case, the heart of this proof involves showing that the sequence (\ref{MVOmega}) is exact for $p = 0$ and $p = 1$. The result then follows from the zigzag lemma (for more details see \cite{Lee}). The exactness is clear at all steps. The difference here from the smooth case is that all functions on graphs are "smooth," so the surjectivity of $i^* - j^*$ is clear: for $f \in \Omega^p(A \cap B)$, we have $(i^* - j^*)(f, 0) = f$.
\end{proof}

\section{Future work}
The following are open questions and new directions that we plan to address on a subsequent manuscript.
\subsection*{Higher dimensions} We would like to extend these results to simplicial and CW-complexes. In particular, the automorphism group action can be extended to cohomology groups of higher degree. We expect to have a combinatorial interpretation of such action, and a higher dimensional version of Theorem \ref{splittingthm}.

\subsection*{Generalizing Theorem \ref{splittingthm}}
If one removes the technical condition from Theorem \ref{splittingthm}, i.e. that $\calH^0(F(\Gamma))$ is trivial, then the product structure of $\Aut(\Gamma)$ in terms of $\ker(\phi)$ and $\calH^1(\Gamma)$ becomes more complicated. We plan to characterize this structure for a more general class of graphs.

\subsection*{$\mathcal H$ and the monoidal structure of Gr}
We intend to determine the compatibility between $\mathcal H$ and different graph operations, including the cartesian product, the wreath product and the tensor product of graphs. In particular, the wreath product of cyclic groups and their representations  is connected to rooted trees \cite{Orellana}. We plan to have a category theoretical interpretation of this correspondence by using $\mathcal H$.

\subsection*{Aut$(\Gamma)$ and quantum mechanics} Theorem \ref{splittingthm} describes a natural splitting of the automorphism group of the graph $\Gamma$, that can be lifted to the space of states of the quantum system on $\Gamma$. We intend to describe the role of $\mathcal H$ in the supersymmetric version of quantum mechanics on graphs.

\subsection*{Natural Orientations}
While the action of the automorphism group on the cohomology is orientation independent, one can ask whether one can find an orientation that yields a ``natural" basis for the cohomology. We define on such notion, one where the basis for the cohomology is given by edge forms with nonnegative entries, i.e. only 0s and 1s.

\begin{defn}
The oriented cycle graph $C_n$ is called {\textit{naturally oriented}} if all of its edges point in the same direction (in the sense induced by an orientation). An arbitrary graph $\Gamma$ is called {\bf{naturally oriented}} if $H^1(\Gamma)$ has a basis whose supports are isomorphic (in the orientation-preserving sense) to naturally oriented cycle graphs. We call the supports of these elements {\bf{primitive cycles}}.
\end{defn}

\begin{figure}[h]
    \centering
    \includegraphics[scale = .35]{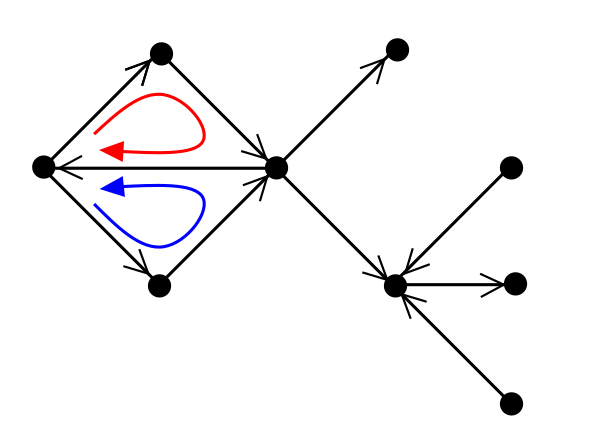}
    \caption{A naturally oriented graph $\Gamma$. Note that the orientation of $F(\Gamma)$ is arbitrary in order for $\Gamma$ to be naturally oriented.}
    \label{naturalori}
\end{figure}

As an example, consider the graph $\Gamma$ shown in Figure \ref{naturalori}. It has $b_1 = 2$, and its primitive cycles are indicated by the red and blue arrows, respectively. In the coordinates induced by the orientation, the red cycle is given by
\begin{align*}
    c_1 = (1,1,1,0,0,0,0,0,0,0)
\end{align*}
and the blue cycle is given by
\begin{align*}
    c_2 = (0,0,1,1,1,0,0,0,0,0).
\end{align*}
Note $H^1(\Gamma) = \text{Span}(c_1, c_2)$. On the other hand, the cycle $c_3$ given by
\begin{align*}
c_3 = c_1 - c_2 = (1,1,0,-1,-1,0,0,0,0,0)
\end{align*}
has support equal to the outer loop, isomorphic to a non-naturally oriented $C_4$. Note that $c_3$ has both positive and negative entries.

We make the following conjecture:

\begin{conj}
Every finite graph has a natural orientation.
\end{conj}

This is trivially true for graphs with $b_1 = 0$ or $1$, and can be seen without too much difficulty to be true for graphs with $b_1 = 2$. For graphs with higher Betti numbers, it is not difficult to construct natural orientations for individual graphs with some trial and error (though these orientations do not always have their primitive cycles being the shortest or most obvious cycles). However, the general case is unknown.

Another unresolved question is that of counting the number of natural orientations on a graph $\Gamma$. This is again fairly easy for low Betti numbers, but the general case is unknown.
\subsection*{Conflict of Interest} On behalf of all authors, the corresponding author states that there is no conflict of interest.

\end{document}